\documentclass[11pt,a4paper]{amsart}
\usepackage[margin=1in]{geometry}
\usepackage[T1]{fontenc}
\usepackage{amsmath,amssymb,amsthm}
\usepackage{tikz}
\usetikzlibrary{arrows.meta}
\usepackage[hidelinks]{hyperref}
\usepackage{cleveref}
\newtheorem{theorem}{Theorem}
\newtheorem{lemma}[theorem]{Lemma}
\title{Four collinear points or six visible points}
\author{Stijn Cambie}
\thanks{Department of Computer Science, KU Leuven Campus Kulak-Kortrijk, 8500 Kortrijk, Belgium. Supported by a postdoctoral fellowship from the Research Foundation Flanders (FWO), grant number 1225224N}
\date{}
\hypersetup{pdftitle={Four collinear points or six visible points},
pdfsubject={A short proof that N(4,6) is at most 880}}
\begin{document}
\pagestyle{plain}
\begin{abstract}
Every finite planar set of at least $880$ points contains four collinear
points or six pairwise visible points. 
\end{abstract}
\maketitle
\section{Introduction}

Let $P$ be a finite set of points in the plane. Two distinct points of $P$
are \emph{visible} if the open segment joining them contains no point of
$P$. The \emph{visibility graph} $G(P)$ has vertex set $P$ and an edge
for each visible pair. A point on the open segment joining a nonvisible
pair is called a \emph{blocker} of that pair.
For integers $s,t\ge3$, let $N(s,t)$ be the least integer, if it exists,
such that every finite planar set of at least $N(s,t)$ points contains
$s$ collinear points or $t$ pairwise visible points.

K\'ara, P\'or, and Wood~\cite{KPW} introduced the
\emph{big-line--big-clique conjecture} in 2005: $N(s,t)$ is finite for
every fixed $s,t$. The case $N(3,t)=t$ is immediate, since without three
collinear points every pair is visible. Also $N(s,3)=s$: every
noncollinear finite set contains a triangle of minimum positive area,
whose vertices are pairwise visible.
K\'ara, P\'or, and Wood proved finiteness for $t\le4$.
Abel et al.~\cite{ABBCetal} established the case $t=5$ for every $s$,
and Bar\'at et al.~\cite{Barat} later obtained the quadratic bound
$N(s,5)\le328s^2$ through their theorem on empty pentagons.
In 2026, Bonnet~\cite{Bonnet} proved $N(4,6)\le10^{11055931}$,
settling the next open case. We give a short proof that
\[
                         N(4,6)\le880,
\]
but do not prove finiteness of other values of $N(s,t)$.

A related, weaker question (attributed to Jan~K\'ara in~\cite{AW10}) replaces the clique number by the chromatic
number: for fixed $s\ge3$ and $c\ge2$, is the size of $P$ bounded if $P$ has no
$s$ collinear points and $G(P)$ can be properly colored with $c$ colors?
This would follow from $N(s,c+1)<\infty$, since a $c$-colorable graph
contains no $K_{c+1}$. Hujter and Kisfaludi-Bak~\cite{HKB} proved this
chromatic version for $s=4$ and $c=5$, giving an upper bound of $2310$
points. More precisely, their Theorem~15 gives $5h(6)-5$, where $h(6)$
is the least number of points in general position that forces an empty
convex hexagon. Here \emph{general position} means that no three points
are collinear, and \emph{empty} means that the convex hull contains no
other point of the set.
Heule and Scheucher's computer-assisted theorem~\cite{HS} gives
$h(6)=30$ (see also~\cite{Suber} for a formal verification). Thus we use the following consequence:
\begin{quote}
A finite planar set with no four collinear points and a $5$-colorable
visibility graph has at most $145$ points.
\end{quote}
The proof, presented in the next section, is elementary (modulo the referenced results).

The bound of $880$ can be improved, but because of elegance and since we have no good guess about the exact value of $N(4, 6)$, somewhere between $16$ and $600$, we kept the simple arguments.

\newpage

\section{Proof}

For a graph $H$, write $e(H)$ for its number of edges and $N_H(v)$ for
the set of neighbors of a vertex $v$.

\begin{lemma}\label{graph}
If a $K_6$-free graph $H$ on $r\ge1$ vertices requires at least $q$ vertex deletions
to become $5$-colorable, where $0\le q\le r$, then
\[
             e(H)\le\frac{2r^2}{5}-\frac{q(r-q)}{10}.
\]
\end{lemma}
\begin{proof}
We start building a maximal clique vertex by vertex. Set $R_1=V(H)$.
For each $i$ with $R_i\ne\varnothing$, choose $v_i\in R_i$ of largest
degree in the original graph $H$, and put
$R_{i+1}=R_i\cap N_H(v_i)$.
Stop when $R_{k+1}=\varnothing$. Each chosen vertex is adjacent to all
earlier choices, and the stopping condition makes
$\{v_1,\ldots,v_k\}$ a maximal clique; see Figure~\ref{greedy}
(cf.~\cite[Section~2]{BN}).

Put $A_i=R_i\setminus R_{i+1}$. Thus $A_i$ consists of the vertices
whose first nonneighbor in the sequence is $v_i$, counting $v_i$ itself
as a nonneighbor. The sets $A_i$ partition $V(H)$.
Let $a_i=|A_i|$, $d_i=d_H(v_i)$, and
$\varepsilon_i=r-a_i-d_i\ge0$.
Each vertex of $A_i$ was available when $v_i$ was chosen, so
\[
       2e(H)\le\sum_i a_i d_i
       =r^2-\sum_i(a_i^2+a_i\varepsilon_i).
\]
If $k\le4$, Cauchy's inequality gives $\sum_i a_i^2 \ge \frac {r^2}4$ and thus $e(H)\le3r^2/8$, which suffices
since $q(r-q)\le r^2/4$. Thus assume $k=5$.
Deleting the vertices which are not adjacent to at least two of the clique vertices $\{v_1,v_2,\ldots, v_5\}$ leaves five
independent sets: an edge between two vertices missing only $v_i$ would
complete a $K_6$ with the other four clique vertices. Since
$\sum_i\varepsilon_i$ counts all misses beyond the first, it is at least the number of deleted vertices, and hence at least
$q$. Choose $0\le x_i\le\varepsilon_i$ with $\sum_i x_i=q$. Then by $\sum_i x_i^2 \le (\sum_i x_i)^2$ and Cauchy-Schwarz again
\[
 \sum_i(a_i^2+a_i\varepsilon_i)
 \ge\sum_i(a_i+x_i/2)^2-\frac14\sum_i x_i^2
 \ge\frac{(r+q/2)^2}{5}-\frac{q^2}{4}
 =\frac{r^2+q(r-q)}5. \qedhere
\]
\end{proof}

\begin{figure}[htbp]
\centering
\begin{tikzpicture}[x=1cm,y=1cm,font=\small,
  region/.style={draw,rounded corners=2pt,line width=.5pt},
  point/.style={circle,fill=black,inner sep=1.7pt}]
  \path[region,fill=black!3] (0,0) rectangle (13,2.6);
  \path[region,fill=black!8] (3.4,.3) rectangle (12.7,2.3);
  \path[region,fill=black!14] (6.8,.6) rectangle (12.4,2.0);
  \path[region,dashed,fill=white] (10.2,.9) rectangle (12.1,1.7);
  \node[anchor=north west] at (.12,2.55) {$R_1=V(H)$};
  \node[anchor=north west] at (3.52,2.25) {$R_2$};
  \node[anchor=north west] at (6.92,1.95) {$R_3$};
  \node at (11.15,1.3) {$R_4,\ldots$};
  \node[point,label=above:$v_1$] at (1.7,1.25) {};
  \node[point,label=above:$v_2$] at (5.1,1.25) {};
  \node[point,label=above:$v_3$] at (8.5,1.25) {};
  \node at (1.7,.65) {$A_1$};
  \node at (5.1,.65) {$A_2$};
  \node at (8.5,.82) {$A_3$};
\end{tikzpicture}
\caption{Building a single clique. Choose $v_i$ from $R_i$ by comparing
  degrees in $H$, then keep only its neighbors as candidates for the next
  choice. The discarded part $A_i=R_i\setminus R_{i+1}$ includes $v_i$.
  The regions show sets of vertices, not their positions in the plane.}
\label{greedy}
\end{figure}

\begin{lemma}\label{deletion}
Let $P$ be a finite set of points in the plane, with $|P|\ge 4$ and no four collinear
points. If deleting at most four vertices from $G(P)$ makes it
$5$-colorable, then $|P|\le439$.
\end{lemma}
\begin{proof}
Enlarge the deleted set $Z$ to four points, $\{z_1,z_2,z_3,z_4\}.$
Choose a supporting line containing an edge of $\operatorname{conv}(Z)$.
It contains two or three points
of $Z$. Split the halfplane containing the remaining points by a second
line through them; if only one remains, choose this line to avoid
$P\setminus Z$. This gives three closed convex cells $C_1, C_2, C_3$ covering the plane. See~\cref{fig:three-cell-decomposition} for a depiction.
Convexity retains every undeleted blocker within a cell, and each cutting
line contains at most one point of $P\setminus Z$.
A deleted point cannot block two retained points in one cell:
both endpoints would have to lie on that cutting line.
The visibility graph of the retained set in each cell is therefore the
subgraph of $G(P)$ induced by those vertices. It is $5$-colorable and
hence has at most $145$ vertices. Thus
$|P\setminus Z|\le3\cdot145$, proving the desired bound.
\end{proof}

\begin{figure}[htbp]
\centering
\begin{tikzpicture}[
  x=.9cm,y=.9cm,font=\small,
  cut/.style={draw=black,line width=.65pt},
  deleted/.style={circle,draw=black,fill=white,line width=.7pt,
                  inner sep=2.5pt},
  retained/.style={circle,fill=black,inner sep=1.4pt}]

\def\threecells{%
  \fill[black!4] (-3,-1.2) rectangle (3,0);
  \fill[black!8] (-3,0)--(.15,0)--(.982,3.2)--(-3,3.2)--cycle;
  \fill[black!15] (.15,0)--(3,0)--(3,3.2)--(.982,3.2)--cycle;
  \draw[cut,<->] (-3.12,0)--(3.12,0);
  \draw[cut,->] (.15,0)--(.9976,3.26);
  \node[anchor=north east] at (3.12,-.06) {$\ell_1$};
  \node[anchor=west] at (1.08,3.13) {$\ell_2$};
  \node at (-.3,-.8) {$C_1$};
  \node at (-1.75,2.72) {$C_2$};
  \node at (2.15,2.72) {$C_3$};
}

\begin{scope}
  \threecells
  \node at (0,3.73) {(a) Two deleted points on $\ell_1$};
  \node[deleted,label=below:$z_1$] at (-2,0) {};
  \node[deleted,label=below:$z_2$] at (2.05,0) {};
  \node[deleted,label=left:$z_3$] at (.384,.9) {};
  \node[deleted,label=left:$z_4$] at (.813,2.55) {};
  \foreach \p in {(-.75,0),(.566,1.6),(-1.85,1.35),
                   (2.15,1.2),(1.4,-.65)}
    \node[retained] at \p {};
\end{scope}

\begin{scope}[xshift=7cm]
  \threecells
  \node at (0,3.73) {(b) Three deleted points on $\ell_1$};
  \node[deleted,label=below:$z_1$] at (-2,0) {};
  \node[deleted,label=below:$z_2$] at (-.6,0) {};
  \node[deleted,label=below:$z_3$] at (2.05,0) {};
  \node[deleted,label=left:$z_4$] at (.553,1.55) {};
  \foreach \p in {(-2.25,1.6),(-.75,2.4),(2.05,1.05),
                   (1.55,2.5),(-2,-.65)}
    \node[retained] at \p {};
\end{scope}

\end{tikzpicture}
\caption{The three-cell decomposition in Lemma~2.}
\label{fig:three-cell-decomposition}
\end{figure}

\newpage
\begin{lemma}\label{quarters}
Let $m$ be a positive integer, and let $P$ be a set of $4m$ points in
the plane with no four collinear points. There exist two intersecting
lines, neither containing a point of $P$, such that each of the four
resulting open regions contains exactly $m$ points of $P$.
\end{lemma}
\begin{proof}
First choose a line containing no point of $P$ that separates $P$ into
sets $A,B$ of size $2m$.
Next, we apply a strategy inspired by a proof of the discrete Ham sandwich theorem, using a standard bisection construction which also appears in~\cite{Nimrod}.
For a unit vector $u$, let $I_A(u)$ be the interval between the two middle
projections of $A$ onto $u$, and define $I_B(u)$ similarly.
The difference of their midpoints is continuous and changes sign under
$u\mapsto-u$, so the midpoints agree in some direction.
Both intervals cannot collapse there, since that would put four points
on one line.

If one interval collapses, its point lies in the interior of the other. By continuity, a sufficiently small change of direction making all projections distinct gives overlapping interiors.
A threshold in the
overlap bisects both sets and avoids all points. Its line crosses the
first, since a parallel line cannot bisect both of its sides.
\end{proof}


\begin{proof}[Proof that $N(4,6)\le880$]
Suppose $P$ has at least $880$ points, no four collinear points, and no
visible $K_6$. Rotate the coordinate axes so that all points of $P$
have distinct horizontal coordinates, and retain the $880$ leftmost
points. 
Visibility
between retained points is unchanged. We may therefore assume
$|P|=880=4m$, where $m=220$.

Let $Q_1,\ldots,Q_4$ be the sets of points in the four regions supplied
by Lemma~\ref{quarters}, in cyclic order, and put $H_i=Q_i\cup Q_{i+1}$,
with indices modulo four. Each $Q_i$ is the intersection of $P$ with a
convex region, and each $H_i$ is its intersection with a halfplane.
The segment between two points of a convex region stays in that region;
hence all its blockers are retained. The visibility graphs of these
eight sets are therefore induced subgraphs of $G(P)$. For an arbitrary subset $U \subset V$, write $b(U)$ and
$g(U)$ for the number of nonvisible (\textbf{b}locked) and visible (\textbf{g}ood) pairs.

Every nonvisible pair $x,y$ has a unique blocker $z$, and the two
pairs $x,z$ and $z,y$ are visible.
If $x$ and $y$ lie in one region, their pair is counted in two halfplanes
$H_i$, and both pairs through $z$ stay in that region. If they lie in
adjacent regions, their pair is counted once, and one of the pairs
through $z$ stays in a single region. Pairs in opposite regions contribute
nothing. Distinct collinear
triples use disjoint pairs, since there are no four collinear points.
Consequently
\[
 \sum_{i=1}^4 b(H_i)\le\sum_{i=1}^4g(Q_i),
 \qquad
 \sum_{i=1}^4\bigl(b(H_i)+b(Q_i)\bigr)\le \sum_{i=1}^4\bigl(g(Q_i)+b(Q_i)\bigr)    =4\binom m2.
\]

Every $Q_i$ has more than $145$ points, so requires at least one deletion
for $5$-colorability. Every $H_i$ has $440$ points, so requires at least
five by Lemma~\ref{deletion}. Put $\beta(r)=\binom r2 - \frac 25 r^2=r^2/10-r/2$.
Lemma~\ref{graph} gives, for each $i$,
\[
 b(Q_i)+b(H_i)
 \ge\beta(m)+\beta(2m)+\frac{m-1+5(2m-5)}{10}
 =\binom m2+\frac{m-26}{10}>\binom m2.
\]
Summing yields the contradiction.
\end{proof}

\section*{Acknowledgments and use of AI}

Marius Tiba presented the visibility problem at the MATRIX--MFO Tandem
Workshop \emph{Combinatorial Interchange}. The bound and its proof were
developed through successive revisions with assistance from GPT-5.6 and
GPT-6. The author checked the whole proof carefully, had a look at the references (checking the correspondence of statements) and did minor editing. 
Open reviews from xsci.com were taken into account during a revision.
A Lean formalisation of~\cite[Thm.~15]{HKB} was communicated to S.~Kisfaludi-Bak.


\end{document}